%% file: manuscript.tex
\documentclass{itor}

\usepackage{natbib}%
\usepackage[figuresright]{rotating}

\usepackage{url}
\usepackage{algorithmic}
\usepackage{algorithm}

\usepackage{amsmath,amsthm}

\theoremstyle{definition}
\newtheorem{example}{Example}

\theoremstyle{remark}

\jname{International Transactions in Operational Research}% no need to specify for ITOR
\jvol{XX}
\jyear{20XX}
\doi{xx.xxxx/itor.xxxxx}
\newtheorem{claim}{Claim}

\usepackage{mathtools}
\usepackage{cleveref} % for \cref, \Cref
\usepackage{verbatim} % for \begin{comment}
\usepackage{tikz}
\usetikzlibrary{patterns}

\newif\iffinalversion
\finalversiontrue
\newcommand{\newtext}[1]{\iffinalversion%
#1%
\else%
\textcolor{blue}{#1}%
\fi%
}
\newcommand{\oldtext}[1]{\iffinalversion%
\else%
\textcolor{red}{#1}%
\fi%
}%

\newlength{\hatchspread}
\newlength{\hatchthickness}
\newlength{\hatchshift}
\newcommand{\hatchcolor}{}
\tikzset{hatchspread/.code={\setlength{\hatchspread}{#1}},
         hatchthickness/.code={\setlength{\hatchthickness}{#1}},
         hatchshift/.code={\setlength{\hatchshift}{#1}},% must be >= 0
         hatchcolor/.code={\renewcommand{\hatchcolor}{#1}}}
\tikzset{hatchspread=3pt,
         hatchthickness=0.4pt,
         hatchshift=0pt,% must be >= 0
         hatchcolor=black}
\pgfdeclarepatternformonly[\hatchspread,\hatchthickness,\hatchshift,\hatchcolor]% variables
   {custom north west lines}% name
   {\pgfqpoint{\dimexpr-2\hatchthickness}{\dimexpr-2\hatchthickness}}% lower left corner
   {\pgfqpoint{\dimexpr\hatchspread+2\hatchthickness}{\dimexpr\hatchspread+2\hatchthickness}}% upper right corner
   {\pgfqpoint{\dimexpr\hatchspread}{\dimexpr\hatchspread}}% tile size
   {% shape description
    \pgfsetlinewidth{\hatchthickness}
    \pgfpathmoveto{\pgfqpoint{0pt}{\dimexpr\hatchspread+\hatchshift}}
    \pgfpathlineto{\pgfqpoint{\dimexpr\hatchspread+0.15pt+\hatchshift}{-0.15pt}}
    \ifdim \hatchshift > 0pt
      \pgfpathmoveto{\pgfqpoint{0pt}{\hatchshift}}
      \pgfpathlineto{\pgfqpoint{\dimexpr0.15pt+\hatchshift}{-0.15pt}}
    \fi
    \pgfsetstrokecolor{\hatchcolor}
    \pgfusepath{stroke}
   }
\usepackage{booktabs} % for \cmidrule (better than cline + extracolspace)
\usepackage{bm} % for correct font and emphasis in formulation max/min

\makeatletter
\newcommand{\specialcell}[1]{\ifmeasuring@#1\else\omit$\displaystyle#1$\ignorespaces\fi}
\newcommand{\bestcolumnemph}[1]{\textbf{#1}}

\usepackage{amssymb}
\usepackage{graphicx}
\usepackage{subfloat}

\usepackage{xcolor}

\DeclarePairedDelimiter{\floor}{\lfloor}{\rfloor}

\newcommand{\myproblem}{C2GCP}
\newcommand{\modelBCE}{BCE}
\newcommand{\modelFMT}{FMT}
\newcommand{\modelBecker}{BBA}
\newcommand{\modelGrid}{MLB}
\newcommand{\modelHierarchical}{MM1}
\newcommand{\modelImplicit}{MM2}
\newcommand{\modelOrigami}{MM3}

\begin{document}

\title{Comparative analysis of mathematical formulations for the two-dimensional guillotine cutting problem}%\thanks{This study was financed in part by the Coordenação de Aperfeiçoamento de Pessoal de Nível Superior - Brasil (CAPES) - Finance Code 001}}

%Henrique Becker\inst{1} \and
%Mateus Martin\inst{2}\orcidID{0000-0002-6722-7571} \and
%Olinto Araujo\inst{3}\orcidID{0000-0003-1136-5032} \and
%Luciana S. Buriol\inst{1}\orcidID{0000-0002-9598-5732} \and
%Reinaldo Morabito\inst{4}\orcidID{0000-0002-3948-305X}}
%Authors, affiliations address.
\author[H. Becker et al.]{Henrique Becker\affmark{a,$\ast$}, Mateus Martin\affmark{b}, Olinto Araujo\affmark{c}, Luciana S. Buriol\affmark{a}, and Reinaldo Morabito\affmark{d}}

\affil{\affmark{a}Department of Theoretical Computer Science, Federal University of the Rio Grande do Sul (UFRGS),\\Porto Alegre (91 501-970), Brasil}
\affil{\affmark{b}Institute of Science and Technology, Federal University of S\~ao Paulo (UNIFESP), S\~ao Jos\'e dos Campos (12 247-014), Brazil}
\affil{\affmark{c}Industrial Technical College, Federal University of Santa Maria, Santa Maria, Brasil}
\affil{\affmark{d}Department of Production Engineering, Federal University of São Carlos (UFSCar), São Carlos (13 565-905), Brazil}
\email{hbecker@inf.ufrgs.br [Henrique Becker]; mpmartin@unifesp.br [Mateus Martin]; buriol@inf.ufrgs.br [Luciana S. Buriol]; olinto@ctism.ufsm.br [Olinto Araujo]; morabito@ufscar.br [Reinaldo Morabito]}

%Author's notes/correspondence etc.
\thanks{\affmark{$\ast$}Author to whom all correspondence should be addressed (e-mail: hbecker@inf.ufrgs.br).}

%%Date
\historydate{Received DD MMMM YYYY; received in revised form DD MMMM YYYY; accepted DD MMMM YYYY}

%Abstract NOTE: max 150 words (ITOR)
\begin{abstract}
About ten years ago, a paper proposed the first integer
linear programming formulation for the constrained two-dimensional
guillotine cutting problem (with unlimited cutting stages).
Since, six other formulations followed, five of them
in the last two years. This spike of interest gave no opportunity for
a comprehensive comparison between the formulations. We
review each formulation and compare their empirical results over
instance datasets of the literature. We adapt
most formulations to allow for piece rotation. The
possibility of adaptation was already predicted but not realized by
the prior work. The results show the dominance of pseudo-polynomial
formulations until the point instances become intractable by them,
while more compact formulations keep achieving good primal solutions.
Our study also reveals a small but consistent advantage of the Gurobi
solver over the CPLEX solver in our context; that the choice of solver
hardly benefits one formulation over another; and a mistake in the
generation of the T instances, which should have the same optima with
or without guillotine cuts.
Our study also proposes hybridising the most recent formulation with a prior formulation for a restricted version of the problem.
The hybridisations show a reduction of about 20\% of the branch-and-bound time thanks to the symmetries broken by the hybridisation.
\end{abstract}

%Keywords, etc.

\maketitle
%% END OF COPIED EXCERPT FROM ITOR TEMPLATE

\input{introduction.tex}

\input{methods.tex}

\input{results.tex}
\input{hybridisation.tex}

\input{conclusions.tex}

\section{Acknowledgments}

This study was financed in part by the Coordenação de Aperfeiçoamento de Pessoal de Nível Superior - Brasil (CAPES) - Finance Code 001

\bibliographystyle{itor.bst}
\bibliography{manuscript}

\end{document}

%% file: introduction.tex
\section{Introduction}

The Constrained Two-dimensional Guillotine Cutting Problem ({\myproblem}) deals with a single rectangular object of size $L \times W$ (also known as \emph{original plate}) and a set $I=\{1,\ldots,m\}$ of rectangular piece types.
Each piece type $i \in I$ is characterized by its size $l_i \times w_i$, profit $p_i$, and maximum number of copies $u_i$ to be produced. 
Since the object cannot accommodate all the pieces, the objective is to select and cut the most valuable subset of pieces.
The {\myproblem} considers that: 
(i) all cuts must be of orthogonal guillotine-type, i.e., the cutting of a rectangle always generates two smaller sub-rectangles (edge-to-edge cuts), without limiting the number of guillotine stages; 
(ii) it produces up to $u_i$ copies of piece type $i \in I$, wherein $u_i < \floor{L/l_i} \floor{W/w_i}$ for at least one piece type $i \in I$ (constrained pattern).
%The {\myproblem} arises in manufacturing environments that seek high utilization rates on expensive raw materials using several guillotine stages and avoid overproduction due to low demand scenarios for each piece type \cite{martin:2020:ijpr}.
The {\myproblem} is known to be strongly NP-hard \cite{hifi:2004}.
The problem is referred to as unweighted if $p_i := l_i w_i$ for all $i \in I$, and weighted otherwise.
In general, the pieces can be cut from the object in its rotated form $w_i \times l_i$; however, a limitation on the object can require the fixed orientation of pieces.
According to the current typology of the cutting and packing problems, the {\myproblem} belongs to the category of the Single Large Object Placement Problems \cite{wascher:2007}.

In the literature, the approaches usually explore a feature of the {\myproblem}: any solution can be represented by a binary tree, where the root is the object, the leaves are the cut pieces, and the branches are the guillotine cuts.
Two types of search are distinguished to explore this binary tree representation: the top-down and bottom-up strategies.
The former considers that the original and residual objects are successively cut towards the pieces.
The latter considers that the pieces are successively merged towards a larger rectangle that fits in the object. 
These strategies have been embedded into branch-and-bound (B\&B) algorithms.
Top-down B\&B algorithms were proposed in \cite{cw:1977,nicos:1995:hadji_orthogonal,hifi:2007:zissimopoulos}, where the object is assumed as the root node; they differ regarding the lower bounds used to prune non-promising nodes.
Bottom-up B\&B algorithms were proposed in \cite{wang:1983,vasko:1989,viswanathan:1993,cung:2000,dolatabadi:2012,yoon:2013}, where the root is a dummy node and $m$ first-level nodes represent a single copy of each piece type; they tend to put effort in different anti-redundancy strategies to avoid symmetrical solutions.
Recently, \cite{russo:2020} presented a complete survey on more than 90 approaches for the {\myproblem} and related problems.
In addition to the previous tree search methods, they resumed the solution methods for the {\myproblem} in two more classes: dynamic programming (DP) algorithms and heuristic/meta-heuristic approaches.
The unconstrained version of the {\myproblem} assumes $u_i \geq \floor{L/l_i} \floor{W/w_i}$ for all $i \in I$.
This problem has been addressed by DP algorithms in \cite{gilmore:1966,beasley:1985:guillotine,russo:2014}.
However, the adaptation of these algorithms for the {\myproblem} is not a straightforward task, as including vector $u_i$ in the representation significantly increases the state space.
For that purpose, one promising path is to consider state space relaxations to obtain tight upper bounds \cite{nicos:1995:hadji_orthogonal,velasco:2019}. 
Lastly, there is a wide field of approaches on heuristics and meta-heuristics for the {\myproblem}.
One may cite the works of \cite{oliveira:1990,alvarez:2002:tabu,morabito:2010,fayard:1998,wei:2015} -- see \cite{russo:2020} for a detailed description.
%They often assume lower bounds provided by the 2-stage version of the {\myproblem} (the pieces are cut from the object in two guillotine stages), greedy heuristics to obtain upper bounds, blocks that represented a subset of pieces instead of the original separated pieces, or even more limited patterns.

In this paper, we examine the mixed-integer linear programming (MILP) formulations proposed for the {\myproblem}.
%To the best of our knowledge, 
We are aware of seven recent formulations for the problem \cite{benmessaoud:2008,furini:2016,martin:2020:grid,martin:2020:bottomup,martin:2020:topdown,becker:2021}.
The first formulation was published just over ten years ago and the last five in the last two years.
All these formulations explore different modeling strategies and may inspire the development of new heuristics/meta-heuristics or algorithms based on decomposition.

The main contributions of this paper are
(i) a discussion of the different modeling strategies in Section \ref{sec:02methods}; 
(ii) the computational experiments performed to assess their performance in terms of solution quality and processing time in Section \ref{sec:03results}; and,
(iii) the claim of a flaw in the T instances conjecture about optimal zero-loss guillotine solutions;
(iv) the proposal of hybridised versions of the most recent formulation which break some symmetries and show a consistent improvement in performance.
To the best of our knowledge, there is no approach in the literature comparing the relative computational performance of all these formulations.
We presume to complement the survey of \cite{russo:2020} and beneﬁt OR practitioners interested in approaching the {\myproblem} and related problems. % with a general-purpose MILP software package.
%Note that these formulations tend to solve small or medium problem instances due to their corresponding large number of variables and constraints, despite the substantial improvements of the MILP software packages.

\begin{comment}
\section{The Elsevier article class}

\paragraph{Installation} If the document class \emph{elsarticle} is not available on your computer, you can download and install the system package \emph{texlive-publishers} (Linux) or install the \LaTeX\ package \emph{elsarticle} using the package manager of your \TeX\ installation, which is typically \TeX\ Live or Mik\TeX.

\paragraph{Usage} Once the package is properly installed, you can use the document class \emph{elsarticle} to create a manuscript. Please make sure that your manuscript follows the guidelines in the Guide for Authors of the relevant journal. It is not necessary to typeset your manuscript in exactly the same way as an article, unless you are submitting to a camera-ready copy (CRC) journal.

\paragraph{Functionality} The Elsevier article class is based on the standard article class and supports almost all of the functionality of that class. In addition, it features commands and options to format the
\begin{itemize}
\item document style
\item baselineskip
\item front matter
\item keywords and MSC codes
\item theorems, definitions and proofs
\item lables of enumerations
\item citation style and labeling.
\end{itemize}

\end{comment}
% Section about the T instances

%% file: methods.tex
\section{Method} \label{sec:02methods}

For the sake of explanation, we chose to aggregate some of the formulations in the same paragraph when they share similar modelling strategies.
We seek to highlight how the interpretation of solutions can lead to very different formulations.

\emph{The {\modelBCE} formulation}, proposed for the Guillotine Strip Packing Problem in~\citet{benmessaoud:2008} and adapted for the {\myproblem} in~\citet{martin:2020:grid}, is based on a theorem that characterizes guillotine patterns and uses coordinates at which pieces may be located.
The theorem states that a pattern is of guillotine type if, and only if, for any region (i.e., sub-rectangle) of the object, at least one of the following conditions is satisfied:
(i) this region contains only a single piece;
(ii) the segments of the piece length in this region on the x-axis consist of at least two disjoint intervals;
and, (iii) the segments of the piece width in this region on the y-axis consist of at least two disjoint intervals.
The formulation is compact in the numbers of variables and constraints with \(O(n^4)\) for the GSPP, where \(n\) is the number of pieces to be packed.
This formulation seems to recall the interval-graph approach of \citet{fekete:1997} for the non-guillotine Orthogonal Packing Problem.

\emph{The {\modelFMT} formulation}, proposed in~\cite{furini:2016}, was inspired by a formulation for the one-dimensional cutting stock problem from~\cite{dyckhoff:1981}.
Given pseudo-polynomial time and space, an instance of the G2KP can be transformed into a bipartite directed acyclic (multi)graph; solving a flow-like problem over such graph is equivalent to solving the original G2KP instance.
The two disjoint and independent sets of vertices are (i) the enumerated plate types and (ii) the enumerated cuts over the plate types.
Each cut \oldtext{vertice}\newtext{vertex} has one incoming edge and one or two outgoing edges.
The head of the incoming edge is the plate \oldtext{vertice}\newtext{vertex} that represents the plate being cut.
The tail of each outgoing edge is a plate \oldtext{vertice}\newtext{vertex} representing a plate produced by the cut.
These are all edges that exist in the graph.
If the cut \oldtext{vertice}\newtext{vertex} has only two incident edges, it represents a trim cut, i.e., a cut that only reduces the size of an existing plate without producing a second plate.
If the cut \oldtext{vertice}\newtext{vertex} has three incident edges, it represents a plate cut into two smaller plates.
As the graph is a multigraph, it allows for parallel edges, representing a cut exactly at the middle of a plate generating two copies of the same plate type.
The aforementioned flow-like problem is as follows.
All edges only allow integer amounts to flow between vertices.
The \oldtext{vertice}\newtext{vertex} representing the original plate type is the only one to start with one flow unit (all other vertices start zeroed).
If a plate \oldtext{vertice}\newtext{vertex} receives any flow amount, it can keep any portion of the flow in the \oldtext{vertice}\newtext{vertex} and freely redistribute the remaining flow among its outgoing edges.
If a cut \oldtext{vertice}\newtext{vertex} receives any flow amount, it \emph{multiplies} the amount of flow received by the number of outgoing edges, and \emph{must} relay the exact amount of flow received to \emph{each} of the outgoing edges, e.g., if a cut \oldtext{vertice}\newtext{vertex} receives two units of flow then each outgoing edge receives two units of flow.
If a plate \oldtext{vertice}\newtext{vertex} represents a plate type of the same dimensions as a piece~\(i\), then each unit of flow kept by the \oldtext{vertice}\newtext{vertex} generates a profit~\(p_i\) constrained to a maximum of~\(u_i \times p_i\).
The problem is deciding how the plate vertices will distribute the flow they receive to maximize said profit.
The {\modelFMT} generates models similar to the graph described and which are solved similarly to the flow-like problem mentioned.
\emph{The {\modelBecker} formulation}, proposed recently in \cite{becker:2021}, is based on the {\modelFMT} formulation but further reduce the model size by employing two more reductions:
(i) similar plate sizes that can only lead to obtaining the same piece multisets are conflated to a single plate size;
(ii) if a plate size is too small to be possible to extract two pieces (of any piece type combination) from it, then all outgoing edges are replaced by edges going directly to obtainable piece-sized vertices (the graph is not bipartite anymore).
The formal description of {\modelBecker} is given in~\cref{sec:hybridisation}, as we give technical details of our hybridisation approach there.

\emph{The {\modelGrid} formulation}, proposed in~\citet{martin:2020:grid}, assumes that each solution can be represented by a sequence of horizontal and vertical guillotine cuts over a two-dimensional grid interpretation of object \(L \times W\).
It was inspired by a formulation for the non-guillotine {\myproblem} from~\citet{beasley:1985:nonguillotine}.
In a {\modelGrid} model, a binary variable \(x_{kij}\) represents the allocation of the left-bottom corner of an piece type~\(k \in \{1,\ldots,m\}\) to a point \((i,j)\) on the object, \(0 \leq i \leq L-l_k\), \(0 \leq j \leq W-w_k\).
Taking into consideration the constraints from~\citet{beasley:1985:nonguillotine}, it ensures a constrained pattern and avoids the overlap between any pair of allocated/cut pieces, which is related to a maximum clique problem.
Then it satisfies the guillotine cutting with binary variables for horizontal cuts \(h_{ii\prime j}\), \(0 \leq i < i\prime \leq L\), \(0 \leq j \leq W\), vertical cuts \(v_{ijj\prime}\), \(0 \leq i \leq L\), \(0 \leq j < j\prime \leq W\), and enabled rectangles \(p_{i_1 i_2 j_1 j_2}\), \(0 \leq i_1 < i_2 \leq L\), \(0 \leq j_1 < j_2 \leq W\).
The main concepts involve associating:
(i) the variables \(x_{kij}\), \(h_{ijj\prime}\) and \(v_{ii\prime j}\) by prohibiting horizontal and vertical cuts on allocated pieces and imposing the allocation of the pieces on cut corners;
(ii) the variables \(h_{ii\prime j}\), \(v_{ijj\prime}\) and \(p_{i_1 i_2 j_1 j_2}\) by allowing only horizontal and vertical edge-to-edge cuts in enabled rectangles.
The formulation is pseudo-polynomial in the numbers of variables and constraints with \(O(mLW+L^2W^2)\).
As expected, one can reduce the number of variables and constraints by using the discretization of normal sets or related ones \citet{herz:1972,cw:1977}.
%Part of these modeling strategies was explored in the context of an object with several defects in~\cite{martin:2020:ijpr,martin:2021:ijpr}.

%%%%%%%%%%%%%%%%%%%%%%%%%%%%%%%%%%%%%%%%%%%%%%%%%%%%%%%
%Martin et al. (2020) - bottom-up and top-down
%%%%%%%%%%%%%%%%%%%%%%%%%%%%%%%%%%%%%%%%%%%%%%%%%%%%%%%
\emph{The {\modelHierarchical} and {\modelImplicit} formulations}, proposed in~\citet{martin:2020:bottomup}, are inspired in the bottom-up strategy of successive horizontal and vertical builds of the pieces.
A build envelops two small rectangles to generate a larger rectangle.
For instance, as introduced in~\citet{wang:1983}, the horizontal build of pieces \(l_1 \times w_1\) and \(l_2 \times w_2\) provides a larger rectangle of size \((l_1+l_2)\times \max\{w_1, w_2\}\), and the vertical build provides a larger rectangle of size \(\max\{l_1, _2\} \times (w_1+w_2)\).
Defining an {\modelHierarchical} or {\modelImplicit} model, it requires to previously determine an upper bound \(\bar{n}\) to the maximum number of builds on object \(L \times W\) (e.g., \(\bar{n}=\sum_{i \in I} u_i\)).
The {\modelHierarchical} formulation is pseudo-polynomial as its definition requires an explicit binary tree structure, which is generated by a procedure that considers upper bound \(\bar{n}\) as an input.
This binary tree structure is represented by a set of triplets \((j,j^-,j^+)\), where \(j^-$ and $j^+\) are the left and right child nodes of node \(j\), respectively; the root node \(j=1\) represents the object (i.e., the solution).
Its main concepts involve ensuring:
(i) each node \(j\) of the binary tree structure can represent either a copy of an piece type \(i \in I\) (binary variable \(z_{ji}\)), an horizontal build (binary variable \(x_{jh}\)), a vertical build (variable \(x_{jv}\)), or it is not necessary in the solution;
(ii) the solution represents a guillotine pattern (i.e., a virtual binary tree) by linking the variables \(x_{jo}\), \(o \in \{h,v\}\), of each parent node \(j\) with the variables of its child nodes \(j^-\) and \(j^+\); and,
(iii) the variables \(L_j\) and \(W_j\) are considered to represent, respectively, the length and width of a node \(j\) according to the previous definition of horizontal and vertical builds, over variables \(z_{ji}\) and \(x_{jo}\).
The {\modelImplicit} formulation, however, is compact as it considers the set of binary variables \(y_{jk}\), \(j,k \in \{1,\ldots,\bar{n}-1\}\), \(j<k\), for representing implicitly the binary tree structure.
The {\modelHierarchical} and {\modelImplicit} formulations were first proposed as integer non-linear programs, and then they were linearized through the use of disjunctive inequalities of big-M type.

\emph{The {\modelOrigami} formulation}, proposed in~\citet{martin:2020:topdown}, is inspired in the top-down strategy of successive cuts on the original and residual objects towards the pieces.
It makes use of the binary tree structure initially proposed for the {\modelHierarchical} formulation.
As a consequence, it presumes the same constraints for representing a guillotine pattern (i.e., a virtual binary tree) by linking the variables \(x_{jo}\), \(o \in \{h,v\}\), of each parent node \(j\) with the variables of its child nodes \(j^-\) and \(j^+\).
However, its geometric constraints are non-trivial.
Since variables \(L_j\) and \(W_j\) are no longer in the formulation, the sizes of the residual objects (i.e., nodes of the binary tree structure) are defined according to the decisions taken in the previous residual objects.
Alternatively stated, the decisions of each node~\(j\) take into consideration the previous decisions of all its ancestral nodes up to the root node through disjunctive inequalities of big-M type.

%% file: results.tex
\section{Comparison between prior formulations} \label{sec:03results}

The list of instance datasets used in this section follows.
\begin{description}
	\item [CU/CW] Datasets introduced by~\cite{fayard:1998}. Their names stand for Constrained (demand) and Unweighted/Weighted. They totalize 22 instances: CU1--11 and CW1--11.
	% APT: 10.1016/S0305-0548(00)00095-2
	\item [APT] Dataset introduced by~\cite{alvarez:2002:tabu}. The whole dataset consists of 40 instances (APT10--49), however, we only use the second half (APT30--49) because the first half is for the unconstrained demand variant. The APT30--39 are unweighted and APT40--49 are weighted.
	\item [FMT59] Group of instances assembled by~\cite{furini:2016} with instance subsets from previous datasets. This dataset was left unnamed, so we reference it by the authors' surnames initials and the number of instances. It contains 37 unweighted instances: A1s, A2s, 2s, 3s, STS2s and STS4s~\cite{cung:2000}; A3--5~\cite{hifi:1997}; CHL1s, CHL2s, and CHL5--7~\cite{cung:2000}, CU1--2 (see above); gcut1--12~\cite{beasley:1985:guillotine}; Hchl\(x\)s (\(x \in \{3, 4, 6, 7, 8\}\))~\cite{cung:2000}; OF1--2~\cite{oliveira:1990}; and W/wang20~\cite{wang:1983,fayard:1998}. As well as 22 weighted instances: A1--2~\cite{hifi:1997}; cgcut1--3~\cite{cw:1977}; CW1--3 (see above); Hchl2/Hchl9~\cite{cung:2000}; HH~\cite{herz:1972,hifi:1997}; okp1--5~\cite{fekete:1997}; and STS2/STS4~\cite{tschoke:1995,alvarez:2002:tabu}.
	\item [Easy18] A subset of FMT59 we defined for this section. Its purpose is to reduce the number of runs needed before we discard a formulation from further consideration. The dataset contains: cgcut1--3, gcut1--12, OF1--2, and wang20.
\end{description}

We selected these datasets because they were already employed by the prior work.
For the CU, CW, and APT datasets, with and without rotation, we use the best known lower bounds from~\cite{velasco:2019}.
For the FMT59 dataset, without rotation, \cite{furini:2016} presents every optimal value\footnote{There is only one typo: the optimal value of the okp2 instance is 22502, not 22503.}, but there is no comprehensive source on the best known values for this dataset when rotation is allowed.

\subsection{Experiments setup}
\label{sec:setup}

Every experiment in this work used the following setup. % unless stated otherwise.
The CPU was an AMD\textsuperscript{\textregistered} Ryzen\textsuperscript{TM} 9 3900X
 12-Core Processor %(3.8GHz, cache: L1 -- 768KiB, L2 -- 6 MiB, L3 -- 64 MiB)
and 32GiB of RAM were available. %(2 x Crucial Ballistix Sport Red DDR4 16GB 2.4GHz)
The operating system used was Ubuntu 20.04 LTS (Linux 5.4.0).
Two kernel parameters had non-default values: \texttt{overcommit\_memory = 2} and \texttt{overcommit\_ratio = 95}.
Hyper-Threading was disabled.
Each run executed on a single thread, and no runs executed simultaneously.
The computer did not run any other CPU bound task during the experiments.

The models for the {\modelBecker} and {\modelFMT} formulations were built using the Julia language and the Gurobi solver.
The models for the {\modelBCE}, {\modelGrid}, {\modelHierarchical}, {\modelImplicit}, and {\modelOrigami} formulations were built using C++ and the CPLEX solver.
To homogeinize the experiments, these implementations were used only to built the models and then save them to MPS files.
Each selected combination of formulation, rotation configuration, and instance originated a single MPS file.
A Julia script then executed each MPS file in four different configurations: CPLEX/LP, CPLEX/MIP, Gurobi/LP, and Gurobi/MIP.

The implementation of {\modelBecker} and {\modelFMT} formulations is available at an online repository\footnote{See~\url{https://github.com/henriquebecker91/GuillotineModels.jl/tree/0.5.0}}.
The scripts for (i) saving the {\modelBecker} and {\modelFMT} models as MPS and (ii) solving all MPS files are also availables\footnote{See~\url{https://github.com/henriquebecker91/phd/tree/BMC-1}}.
The implementations of all the other formulations, as well as the script for generating the MPS files, are available upon request to the authors. %at~\url{https://github.com/mateuspmartin/g2slopp/tree/BMC-1} (CHECK WITH MARTIN IF WE ARE GONNA UN-PRIVATE IT NOW).
The same version of the compilers and solvers was used for the MPS generation and the MPS solving phases.
Those are: Julia 1.5.3, g++ 9.3.0, CPLEX 20.1, and Gurobi 9.1.1.
At least for the solvers, these were the latest versions available.

In both CPLEX and Gurobi some non-default configurations were used.
The solvers were configured to:
employ a single thread;
use a specified seed; %(\texttt{CPX\_PARAM\_RANDOMSEED}, in CPLEX, and \texttt{Seed}, in Gurobi, were set to one);
employ an integer tolerance adequate for the instances; %avoid finishing with suboptimal solutions for the selected datasets; % (\texttt{CPX\_PARAM\_EPGAP}, in CPLEX, and \texttt{MIPGap}, in Gurobi, were set to \(10^{-6}\));
and respect an one hour time-limit. % (\texttt{CPX\_PARAM\_TILIM}, in CPLEX, and \texttt{TimeLimit}, in Gurobi, were set to 3600). % For a more graceful handling of memory exhaustion, we also set CPLEX parameter \texttt{CPXPARAM\_MIP\_Limits\_TreeMemory} to 28672 (Gurobi does not seem to provide a similar parameter).
Only when solving the {\modelFMT} and {\modelBecker} formulations, we use the barrier method for solving the LP and for solving the root node relaxation. %(\texttt{CPXPARAM\_LPMethod} and \texttt{CPXPARAM\_MIP\_Strategy\_StartAlgorithm}, in CPLEX, were both set to 4, and \texttt{Method}, in Gurobi, was set to 2).

\subsection{Outline of the experiments}

Each run can be uniquely identified by a combination of instance, formulation, rotation configuration (allow rotation or not), solve mode (MIP or LP), and solver (CPLEX or Gurobi).
The first three characteristics determine a MPS file, the last two determine four distinct runs over the same MPS file.

The whole set of runs consists of:
\begin{enumerate}
\item The Easy18 instances combined with each of the seven considered formulations and both rotation configurations, except by the BCE formulation with rotation enabled, which we did not implement.
\item The CU, CW, and FMT59 instances combined with the {\modelBecker}, {\modelOrigami}, hierachical, and {\modelImplicit} formulations and both rotation configurations.
\item The APT instances combined with the four formulations mentioned above but only with rotation disabled. There were no successful runs with rotation disabled and, therefore, we decided to not spend computational effort in the rotation enabled counterparts. % TODO: check with Martin if we should run the rotation ones for APT.
\end{enumerate}

\subsection{Comparison between solvers}

The goal of this section is to answer two questions:
(i) is one of the solvers superior in our context?
(ii) does a choice of solver benefit a specific formulation?

\Cref{tab:cplex_vs_gurobi} answers the first question by revealing a small but consistent advantage for the Gurobi solver.
Nevertheless, Gurobi does not completely dominates CPLEX, as each solver had some instances only solved by it.
The columns of~\Cref{tab:cplex_vs_gurobi} are:
\emph{\#opt} -- number of runs finished by optimality;
\emph{\#u. opt} -- number of optimal runs unique to the respective solver (i.e., other solver did not reach optimality);
\emph{\#best} -- number of optimal runs in which the respective solver finished before the other solver (counting the ones not finished by the other solver);
\emph{\#c. best} -- number of clean best times, i.e., optimal runs that took at least one minute for the respective solver and either were not solved by the other solver or it took double the time to solve;
\emph{avg. t.} -- mean run time in seconds (runs ended by timeout or memory exhaustion are counted as taking one hour);
\emph{avg. o. t.} -- mean run time of optimal runs in seconds.

\begin{table}[h]
  \center
  \caption{Comparison amongst CPLEX and Gurobi results.}
  \setlength\doublerulesep{0.05\baselineskip}
  \begin{tabular}{lcrrrrrr}
    \hline\hline
    \textbf{Solver} & \textbf{Type} & \textbf{\#opt} & \textbf{\#u. opt} & \textbf{\#best} & \textbf{\#c. best} & \textbf{avg. time} & \textbf{avg. s. time} \\\cmidrule(lr){1-2}\cmidrule(lr){3-8}
     CPLEX & MIP & 288 & 12 &  96 & 16 & 2435.86 & 455.21 \\
    Gurobi & MIP & 302 & 26 & 218 & 63 & 2339.84 & 353.62 \\
     CPLEX & LP  & 704 &  4 & 194 &  6 &  379.17 &  40.62 \\
    Gurobi & LP  & 720 & 20 & 530 & 24 &  297.79 &  31.78 \\\hline\hline
  \end{tabular}
  \label{tab:cplex_vs_gurobi}
\end{table}

\begin{table}[h]
  \center
  \caption{Comparison amongst CPLEX and Gurobi results by formulation.}
  \setlength\doublerulesep{0.05\baselineskip}
  \begin{tabular}{lrrrrrrr}
    \hline\hline
    \textbf{Measure} & \textbf{\modelBCE} & \textbf{\modelBecker} & \textbf{\modelFMT} & \textbf{\modelGrid} & \textbf{\modelHierarchical} & \textbf{\modelImplicit} & \textbf{\modelOrigami} \\\hline
    Optimal & 105.88 &  99.31 & 131.57 & 150.00 & 100.00 & 100.00 & 101.24 \\
    T. Time &  85.45 & 101.75 &  74.71 &  58.61 &  45.09 &  27.68 &  67.60 \\\hline\hline
  \end{tabular}
  \label{tab:percentages_gurobi_cplex}
\end{table}

\Cref{tab:percentages_gurobi_cplex} answers the second question.
It presents the percentage of solved runs and total time spent by Gurobi in relation to CPLEX, broken down by formulation, for all MILP runs.
Therefore, in the first row, figures above 100\% mean Gurobi solved more runs than CPLEX and, in the second row, figures below 100\% mean Gurobi spent less time than CPLEX (runs ended by time or memory limit are counted as taking one hour).
Gurobi have better results for all formulations except the {\modelBecker} formulation in which the results are very similar (only slightly worse).
The choice of Gurobi as solver improve the results for some formulations more than others but, in general, the formulations which solve less instances are the most benefited.
Therefore, we consider Gurobi a fair choice for the rest of the paper

\subsection{Comparison between formulations}

The goal of this section is to provide empirical evidence for the choice of one formulation over other, and to identify the impact of allowing rotation over all formulations.
Given the number of considered formulations, we use \Cref{tab:easy18} to filter the considered formulations further.
The columns of~\Cref{tab:easy18} are similar to the ones present in \Cref{tab:cu_cw} and \Cref{tab:fmt59_apt}; exceptions are noted close to each table.
The explanation of these columns follows:
\#o -- number of runs finished by optimality;
\(g_{lb}\) -- the average percentage gap between the best lower bound found and the best known lower bound (if the run finishes without a solution, as is the case of memory exhaustion, we assume a trivial empty solution was returned);
\(t_s\) -- the average total time spent by a run in seconds (both timeout and memory exhaustion count as one hour);
\(g_{ub}\) -- the average percentage gap between the continuous relaxation and the best known lower bound;
\#f -- the number of runs finished by timeout or memory exhaustion during the root node relaxation phase (these are excluded from \(g_{ub}\)).

\begin{table}[h]
  \center
  \caption{Filtering formulations with EASY18 dataset.}
  \setlength\doublerulesep{0.05\baselineskip}
  \begin{tabular}{lrrrrrrrrrr}
    \hline\hline
    & \multicolumn{5}{c}{Fixed} & \multicolumn{5}{c}{Rotation} \\
    \cmidrule(lr){2-6}\cmidrule(lr){7-11}
    Method & \#o & \(g_{lb}\) & \(t_s\) & \(g_{ub}\) & \#f & \#o & \(g_{lb}\) & \(t_s\) & \(g_{ub}\) & \#f \\
    {\modelBCE} & 2 & 7.00 & 3341 & 7.31 & 0 & -- & -- & -- & -- & -- \\
    {\modelBecker} & 18 & 0.00 & \(>1\) & 1.74 & 0 & 18 & 0.00 & \(>1\) & 0.63 & 0 \\
    {\modelFMT} & 13 & 27.78 & 1336 & 2.48 & 4 & 10 & 44.44 & 1723 & 0.71 & 7 \\
    {\modelGrid} & 10 & 33.62 & 1671 & 3.85 & 4 & 3 & 78.23 & 3002 & 3.22 & 9 \\
    {\modelHierarchical} & 16 & 0.08 & 750 & 7.31 & 0 & 14 & 0.31 & 1082 & 3.32 & 0 \\
    {\modelImplicit} & 10 & 0.33 & 1684 & 7.31 & 0 & 9 & 0.45 & 1885 & 3.32 & 0 \\
    {\modelOrigami} & 16 & 0.07 & 685 & 7.31 & 0 & 12 & 0.43 & 1297 & 3.32 & 0 \\\hline\hline
  \end{tabular}
  \label{tab:easy18}
\end{table}

\begin{table}[h]
  \center
  \caption{Solving datasets CU and CW}
  \setlength\doublerulesep{0.05\baselineskip}
  \begin{tabular}{lrrrrrrrrrrrrrrrr}
    \hline\hline
    & \multicolumn{8}{c}{CU} & \multicolumn{8}{c}{CW} \\
    \cmidrule(lr){2-9}\cmidrule(lr){10-17}
    & \multicolumn{4}{c}{Fixed} & \multicolumn{4}{c}{Rotation} & \multicolumn{4}{c}{Fixed} & \multicolumn{4}{c}{Rotation}\\
    \cmidrule(lr){2-5}\cmidrule(lr){6-9}\cmidrule(lr){10-13}\cmidrule{14-17}
    Alg. & \#o & \(g_{lb}\) & \(t_s\) & \(g_{ub}\) & \#o & \(g_{lb}\) & \(t_s\) & \(g_{ub}\) & \#o & \(g_{lb}\) & \(t_s\) & \(g_{ub}\) & \#o & \(g_{lb}\) & \(t_s\) & \(g_{ub}\) \\
    {\modelBecker} & 10 & 9.09 & 425 & 0.21 & 9 & 18.18 & 716 & 0.06 & 11 & 0.00 & 15 & 1.24 & 10 & 0.00 & 496 & 1.72 \\
    {\modelHierarchical} & 3 & 0.54 & 2928 & 1.45 & 0 & 0.68 & 3600 & 0.57 & 5 & 0.00 & 2560 & 11.13 & 3 & 0.01 & 3052 & 5.26 \\
    {\modelImplicit} & 0 & 0.80 & 3600 & 1.45 & 0 & 0.88 & 3600 & 0.57 & 0 & 0.89 & 3600 & 11.13 & 0 & 0.51 & 3600 & 5.26 \\
    {\modelOrigami} & 3 & 0.78 & 3021 & 1.45 & 2 & 0.97 & 3400 & 0.57 & 5 & 0.00 & 2602 & 11.13 & 2 & 0.80 & 3259 & 5.26 \\\hline\hline
  \end{tabular}
  \label{tab:cu_cw}
\end{table}

\Cref{tab:easy18} shows that, for the EASY18 dataset, {\modelBecker} dominates all other formulations.
The {\modelGrid} has the largest average lower bound gap.
The model size often prevents its runs from finishing solving the root node relaxation.
The same problem is also seen in {\modelFMT} runs but to a smaller extent.
{\modelBCE} solves the least instances, its lower bound gap is smaller than {\modelFMT} and {\modelGrid} but considerably above the rest of the instances.
{\modelHierarchical} and {\modelOrigami} solve most instances and have very small lower bound gaps.
Finally, {\modelImplicit} solves an amount of instances comparable to {\modelFMT} and {\modelGrid} but, different from them, the root node relaxation is always solved and a good primal solution too.

Considering these results, the authors chose to remove {\modelBCE}, {\modelGrid}, and {\modelFMT} from further comparison.
The rationale for these choices follows: {\modelBCE} solves very few instances leading to a great increase in experiment times; the model size of {\modelGrid} leads to memory problems, especially for runs allowing rotation; and {\modelFMT} is similar to {\modelBecker} but without some additional enhancements (see~\Cref{sec:02methods}).

In \Cref{tab:cu_cw}, we see two distinct behaviors emerge.
The {\modelBecker} (pseudo-polynomial) starts to present a behavior similar to {\modelFMT}: either solving the instances faster than the other formulations, or failing to solve the root node relaxation at all\footnote{The table ommits it but {\modelBecker} fails to solve the root node relaxation one time for CU/Fixed and two times for CU/Rotation.}.
The other three formulations have difficulty to prove optimality, however they always solve the root node relaxation and provide primal solutions of good quality.
The \(g_{ub}\) column indicates that {\modelHierarchical}, {\modelImplicit}, and {\modelOrigami} have the same average upper bound gap.
The reason for this similarity is that the three formulations, while distinct, make use of the same additional constraints to tighten the upper bound to a precomputed value.
In all three formulations, these constraints impose a tighter bound than the one imposed by the remainder of the formulation, leading to this similarity.
The problem becomes harder for all formulations if rotation is allowed.
The values in the \(g_{ub}\) column for {\modelHierarchical}, {\modelImplicit}, and {\modelOrigami} reduce when rotation is allowed, however this happens only because their upper bounds stay the same while the best known solution increases in value.

\begin{table}[h]
  \center
  \caption{Solving datasets FMT59 and APT}
  \setlength\doublerulesep{0.05\baselineskip}
  \begin{tabular}{lrrrrrrrrrrrrrrr}
    \hline\hline
    & \multicolumn{10}{c}{FMT59} & \multicolumn{5}{c}{APT} \\
    \cmidrule(lr){2-11}\cmidrule(lr){12-16}
    & \multicolumn{5}{c}{Fixed} & \multicolumn{5}{c}{Rotation} & \multicolumn{5}{c}{Fixed} \\
    \cmidrule(lr){2-6}\cmidrule(lr){7-11}\cmidrule(lr){12-16}
    Method & \#o & \(g_{lb}\) & \(t_s\) & \(g_{ub}\) & \#f & \#o & \(g_{lb}\) & \(t_s\) & \(g_{ub}\) & \#f & \#o & \(g_{lb}\) & \(t_s\) & \(g_{ub}\) & \#f \\
    {\modelBecker} & 57 & 1.69 & 183 & 1.75 & 1 & 56 & 1.05 & 233 & 3.28 & 1 & 0 & 100.00 & 3600 & -- & 20 \\
    {\modelHierarchical} & 27 & 1.00 & 2387 & 4.89 & 0 & 20 & -1.16 & 2657 & 4.66 & 0 & 0 & 11.32 & 3601 & 1.86 & 0 \\
    {\modelImplicit} & 10 & 1.40 & 3015 & 4.89 & 0 & 9 & -1.04 & 3077 & 4.66 & 0 & 0 & 3.10 & 3600 & 1.86 & 0 \\
    {\modelOrigami} & 25 & 1.39 & 2328 & 4.89 & 0 & 13 & -0.74 & 2837 & 4.66 & 0 & 0 & 90.39 & 3509 & 1.90 & 9 \\\hline\hline
  \end{tabular}
  \label{tab:fmt59_apt}
\end{table}

\Cref{tab:fmt59_apt} corroborates the findings of~\Cref{tab:cu_cw}.
{\modelBecker} solves more FMT59 instances, but ends up with a larger \(g_{lb}\) than the other formulations because of the poor solution quality in the few unsolved instances.
For the FMT59 instances, {\modelHierarchical} has the lowest~\(g_{lb}\) but, for the APT instances {\modelImplicit} surpasses it.
The {\modelBecker} is unable to solve the root node relaxation for any of the APT instances during the one hour time limit.
The column~FMT59/Rotation/\(g_{lb}\) has negative values because, as mentioned in~\Cref{sec:03results}, we use the known optima from fixed orientation for this particular dataset.

\subsection{About the T instances from \cite{hopper_thesis}}

In a preliminary experiments phase, the authors have considered the instances of the datasets N and T from \cite{hopper_thesis}. % Eva Hopper's Thesis~
Both datasets have 35 instances each, and were generated by recursively cutting a 200x200 plate while following some rules. %The number of instances (35) comes from 7 distinct \(n\) values (17, 25, 29, 49, 73, 97, 199) times 5 random seeds.

In \cite{hopper_thesis}, the instances were used to evaluate heuristics for the Strip Packing Problem.
The instances have a width of 200 and their optimal height is 200 (a pattern with no waste).
Hence, the optimality gap could be easily computed.

The main distinction between datasets T and N is that: dataset T used only guillotine cuts in the generation process; dataset N also replaced some pieces by the usual 5-piece pattern which is impossible to cut using only guillotine cuts (see the pieces 12, 13, 14, 15, and 16 in \Cref{fig:T1a}).

\begin{figure}[h]
  \center
  \includegraphics[width=0.6\linewidth]{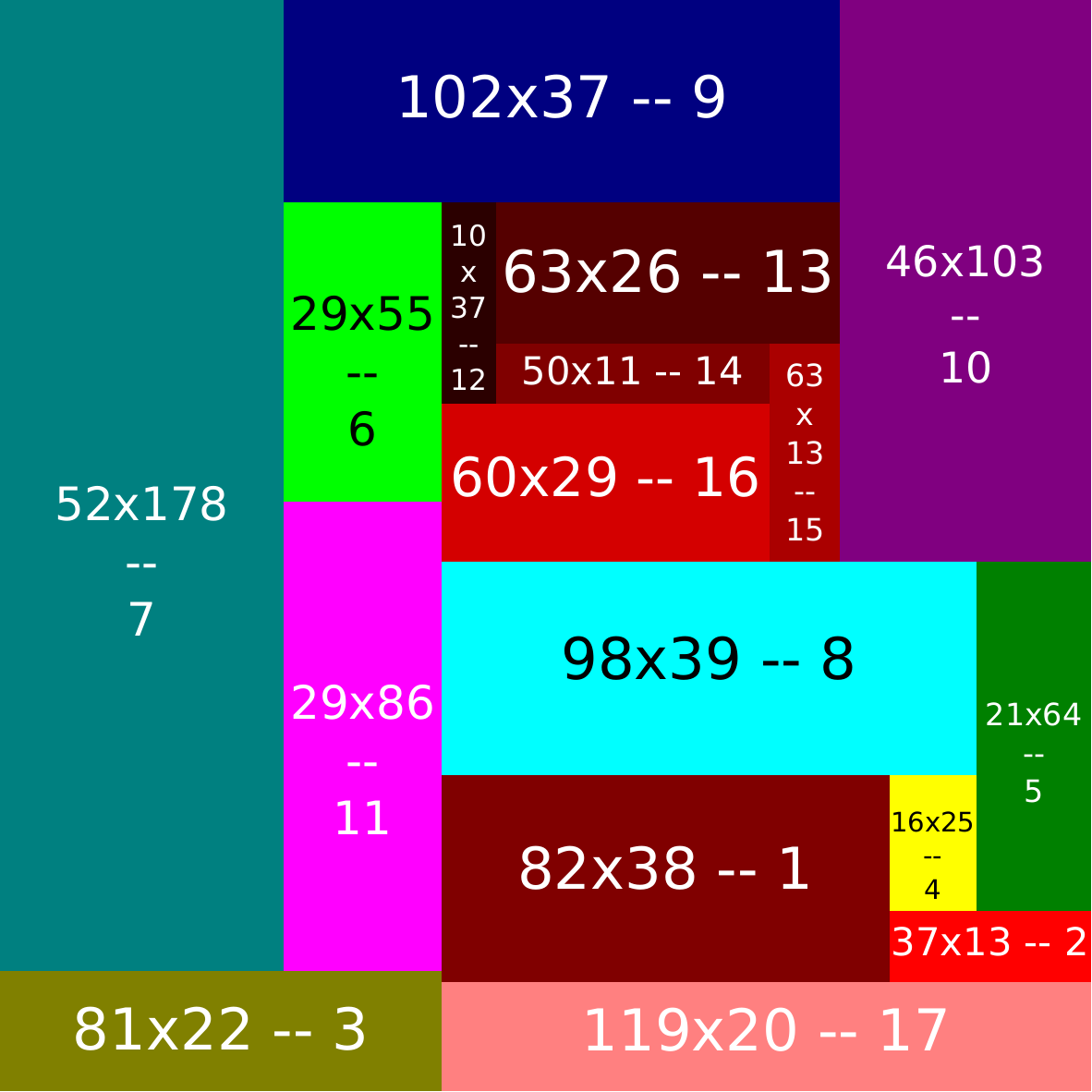}
  \caption{A non-guillotinable optimal solution for instance T1a.}
  \label{fig:T1a}
\end{figure}

\begin{comment}
Instances for which we proved infeasibility:
jl_yMhj5m T1a
jl_meyQRn T1b
jl_IJhWMm T1c
jl_aGVVYp T1d
jl_on4aRl T1e
jl_axD9ro T2c
jl_idk5Cm T2e
jl_61BNyp T3a
\end{comment}

In our preliminary experiments, we adapted the instances to the knapsack problem by adopting \(L = 200\) and \(p_i = 1, \forall i \in 1..n\).
Even with a three hours time limit we optimally solved only 8 instances of dataset T (T1a, T1b, T1c, T1d, T1e, T2c, T2e, T3a) but, for all of them, we obtained a value smaller than the number of pieces in the instance.
This unexpected result led us to manual analysis of the instances which revealed a mistake in their generation.
The instances of the T dataset seem to have been generated using the same generator used for the N instances.
Consequently, there is no guarantee of a wasteless solution using only guillotine cuts.
In fact, for the first instance of the dataset T (T1a), we can provide proof it is actually impossible to obtain a guillotine packing of all pieces.

\begin{claim}
The entire piece set of the instance T1a cannot be packed into a space of 200x200 using only guillotine cuts.
\end{claim}

\begin{proof}

%Therefore, any guillotine cut that leads to some waste cannot lead to the desired pattern.
The desired pattern can only be wasteless as the summed area of the pieces equals to the available area.
Some piece (or piece subset) has to be positioned below/above piece 7 (52x178) for a pattern to have no waste.
The only piece of height 22 is piece 3 (81x22) and no subset of pieces sums up height 22.
As no other pieces sum up exactly height 22, we cannot start cutting the plate using a horizontal cut separating pieces 7 and 3: the space by the sides of piece 3 would have some waste.

The only remaining option is a vertical cut separating pieces 7 and 3 from the rest of the plate.
This cut must create a plate of width 81, so piece 3 can be obtained by means of a subsequent horizontal cut with no waste.
Consequently, the desired pattern must have a subpattern of size 81x179 consisting of piece 7 and any other pieces (except piece 3) and with no waste.
Such subpattern is impossible to obtain.
The pieces able to fit by piece 7 sides are: 6 (29x55) and 11 (29x86), both with the exact width, and also the combination of pieces 4 (16x25) and 15 (13x40), that side-by-side sum width 29.
Other pieces of small width (e.g., 5 and 12) cannot combine in a way that sum width 29.
These pieces are insufficient to fill the gap.
Therefore, there is no guillotinable pattern able to pack the entire piece set within the available area. \qed

\end{proof}

As far as we know, no previous paper has made this claim.
For example, \cite{what_matters} is from 2018 and makes the usual distinction between datasets N and T as non-guillotineable and guillotineable.
The closest claim we could find was from~\cite{wei:2014} (from 2014) which says:
\begin{quote}
``The set T consists of 70 instances whose known perfect packing follow guillotine-cut constraint. However, when we follow the method described by Hopper (2000) to restore the known perfect packing for T instances, we found the known perfect packing clearly does not follow the guillotine-cut constraint.''
\end{quote}
Unfortunately, this claim gets the number of instances wrong, so we are not sure if this was a typo or if they mixed T and N instances.
Much of the prior work that solved instances from T dataset was either focused on heuristic methods or exact non-guillotine methods.
Examples of such works are
\cite{alvarez:2008} (non-guillotine heuristic founds optimal solution for smaller T instances) and % https://doi.org/10.1016/j.cor.2006.07.004
\cite{wei:2011} (non-guillotine heuristic).
We belive this explains why this discrepancy was not noticed sooner.

Some works may have been impacted by the mistake in the dataset generation.
\cite{bortfeldt:2012} % https://link.springer.com/article/10.1007%252Fs10479-012-1084-7
proposes a tree-search algorithm that respects the guillotine constraint.
The work compares optimality gaps of the proposed method and other non-guillotine methods.
The work assumes all methods could reach the best values for such instances, which is not true.
Consequently the optimality gaps are wrong and unfair to the proposed method.
Fortunately, the work does not use only the T dataset so the problem is mitigated.
Other works that appear to be in the same circumnstances are
\cite{thomas:2014} and % https://www.sciencedirect.com/science/article/pii/S0965997814000489?via%3Dihub
\cite{shang:2020}. % https://iopscience.iop.org/article/10.1088/1742-6596/1656/1/012005/meta

Finally, even if our hypothesis about the generation mistake is correct, we cannot claim every instance in dataset T is impossible to fit into a 200x200 space using only guillotine cuts.
The procedure used to generate dataset N (and, in our hypothesis, dataset T too) does not guarantee that guillotine and non-guillotine methods reach the same optimal solution value: the procedure only makes very improbable for it to happen.

%The procedure used to generate dataset N (and, seemingly, dataset T too) involves replacing one or more pieces by five-piece sets.
%These five-piece sets can only fit into the space left by the replaced piece by using a non-guilllotinable pattern.
%However, it may be possible to find a pattern these five-piece sets are not packed together and the whole pattern is guillotinable.
%For example, any non-guillotinable pattern may be be rearranged into a guillotinable pattern if (i) the original plate has both dimension two times larger than the pattern dimensions and (ii) all the other pieces are of unitary size (1x1).

%% file: hybridisation.tex
\section{Hybridisation with the restricted formulation}
\label{sec:hybridisation}

This section proposes another symmetry-breaking change compatible with the FMT and BBA formulations explained in~\cref{sec:02methods}. %This change further complicates the formulation, and the empirical results did not reveal an improvement as large as the previously discussed enhancements. Therefore, the author chose to keep this change self-contained in this chapter.
We are unaware of any previous application of the proposed change to unrestricted 2D guillotine problems.
The Cut-Position enhancement from~\citet{furini:2016} draws inspiration from the same broad idea: to get closer to a formulation for the (simpler) restricted problem while keeping optimality for the unrestricted problem.
However, the proposed change and the Cut-Position both approach this goal in distinct and complementary ways.

\subsection{The restricted problem and piece-outlining cuts}

A guillotine cutting problem is said to be \emph{restricted} if (i) each horizontal (vertical) guillotine cut must match the length (width) of a piece that fits into the plate, i.e., it happens at a \emph{restricted cut position}, and (ii) a piece of that length (width) is guaranteed to be obtained from the first child plate.
The concept of a \emph{restricted} variant appears first in the context of the three-staged guillotine cutting problem.
The two-staged problem is inherently restricted: a cut that does not match the outline of a piece, or a cut that does not guarantee a piece extraction because it is not paired with a cut from the only other stage, is a cut that will not help to obtain any pieces before the two stages are over.
Only when the number of stages is three or more that an optimal solution for the unrestricted problem may require cuts without such immediate purposes.
Applying the concept of \emph{restricted} to unlimited stages is not new, \citet{furini:2016} already does it.
\citet{furini:2016} also presents an intermediary variant which respects (i) but not (ii), this variant can be referred to as \emph{position-only} restricted problem.
%The \emph{position-only restricted problem} is the one solved by the \emph{restricted priced} in~\cref{sec:furini_vs_enhanced_comparison}.

The restricted problem has at least two performance advantages over the unrestricted problem.
The first advantage is related to the number of restricted cut positions: the number of cuts positions in any plate is bounded by the number of pieces (i.e., linear on the input) and not pseudo-polynomial (i.e., bounded by plate dimensions), even if the number of plates themselves is still pseudo-polynomial.
The second advantage is related to the piece extraction requirement.
There is no optimality loss if, after a cut at a restricted position related to a single piece, it is immediately determined that, if necessary, the first child plate will be cut again in the next stage to obtain the respective piece.
The possibility of joining two decision variables together has led previous prior on the restricted problem, as \citet{silva:2010}, to redefine \emph{cut} to mean \emph{one or two guillotine cuts associated a priori to a piece type and which outline and obtain a piece-sized plate that cannot be further cut}.
The guillotine cuts considered until now may incidentally outline and obtain a piece-sized plate as their child plates.
However, they are not a priori associated with a single piece type, nor do they guarantee their first child plate (if piece-sized) cannot be further cut.
In this chapter, the text distinguishes between these two kinds of cuts to avoid confusion.
The single and unassociated cuts considered until now will be referred to as \emph{basic guillotine cuts} (or BGCs for short), and this new definition of cut will be referred to as \emph{piece-outlining cuts} (or POCs for short).
\Cref{fig:piece_outlining_cut} may help to visualise the \emph{piece-outlining cuts}

\begin{figure}[h]
  \caption{Piece-outlining cuts}
  \center
  \input{piece_outlining_cut.tex}
  %\legend{Souce: the author.}
  \label{fig:piece_outlining_cut}
\end{figure}
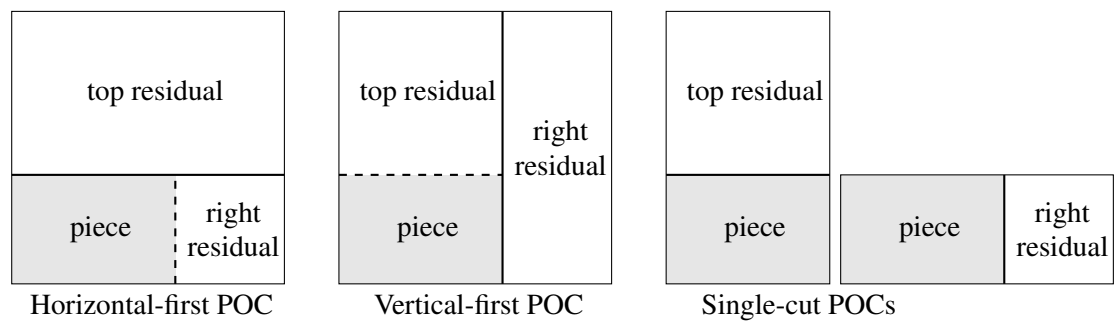

While a POC constituted by two BGCs may be considered a single decision by a solving method and may be seen as happening in succession, in practice, stage restrictions may change the order a cutting machine performs them.
However, these real-world details do not impact the modelling and will not be discussed in this chapter.
Essentially, each piece type that fits into a plate has two POCs associated with it.
One POC that does the horizontal guillotine cut first and then obtains the piece from the first child plate through a vertical cut (if necessary).
This POC always leaves a \emph{top residual plate} (second child plate of the first cut) and often a \emph{right residual plate} (second child plate of the second cut).
The other POC is the same, except that the vertical cut is done first (i.e., always leaving a right residual and often a top residual plate).
Finally, the piece-sized plate obtained by a POC is the first child plate of the second cut if the second cut exists; otherwise, just the first child plate of the only cut.
The piece-sized plate is either immediately regarded as an obtained piece (already enforcing a rule of the restricted problem) or may be considered waste (e.g., the cutting stock problem often allows piece overproduction).
However, the piece-sized plate is \emph{never} treated as an intermediary plate that could be further cut.

A caveat of the coupled representation mentioned above is that, for some instances of the restricted problem, the number of POCs may be larger than the number of restricted cut positions.
In general, each piece type that fits into a plate has two POCs\footnote{The exception happens when the piece type shares the length or the width with the plate and, consequently, both POCs are equivalent and can be considered the same.} (vertical-first and horizontal-first).
An horizontal (vertical) BGC at a restricted position is shared by all piece types with the same length (width).
However, the main advantage of the coupled representation comes from breaking symmetries, not reducing the number of variables.
%For example, if a stripe of width~\(10\) is obtained by a BGC in the restricted problem, it may be used to obtain a single piece of width~\(10\) and four pieces of width~\(8\), and every permutation in the order of the pieces are obtained from the strip is a symmetry, the POC enforce the width~\(10\) piece is the first to be obtained, and that no other mechanism is necessary to guarantee that the piece will be obtained from such plate.

The POCs are a natural choice for the \emph{restricted} problem but not for the \emph{unrestricted} problem for mostly two main reasons.
The first reason is that, in the restricted problem, each horizontal (vertical) cutting position shares length (width) with at least one piece.
However, in the unrestricted problem, some cutting positions can only be reached by combining many pieces.
The second reason is that the definition of the \emph{restricted} problem guarantees that employing only POCs cannot lead to optimality loss; the same is not true for the unrestricted problem (see \cref{fig:distinctions_restricted_unrestricted}).

\begin{figure}[h]
  \caption{Distinctions between, restricted, position-only restricted, and unrestricted problems. The restricted problem cannot obtain the unrestricted optimal solution. If the first cut happens at a restricted position, the child plates cannot fit the six pieces of the optimal solution, regardless of the piece chosen to be obtained first from the original plate and the orientation of the first cut employed. The position-only restricted problem can obtain the unrestricted optimal solution if, by chance, there is an unpacked piece with a width that matches the necessary vertical cut; otherwise, the solution is also out of reach.}
  \center
  \input{distinctions_restricted_unrestricted.tex}
  \label{fig:distinctions_restricted_unrestricted}
\end{figure}
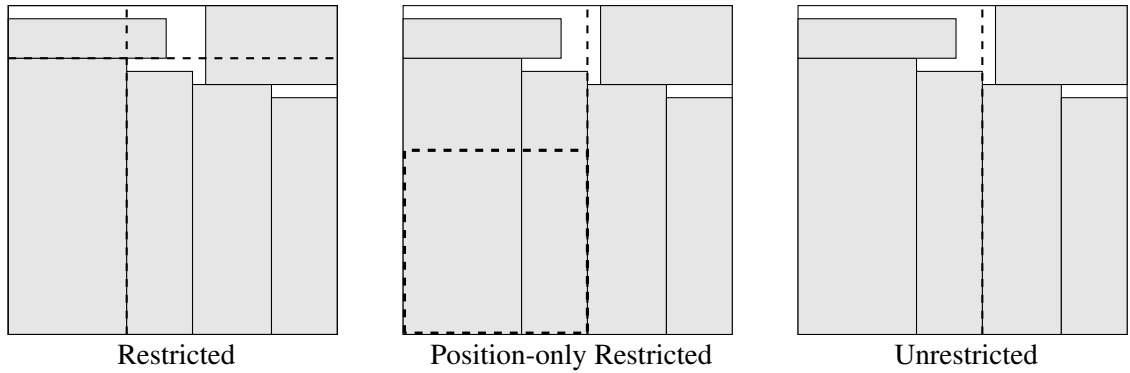

\citet{silva:2010}~proposes a mathematical formulation for the two-stage and three-stage restricted cutting stock problems.
The formulation was not named by its authors; hence, in this text, it will be referred to as SAV (from the author's surname initials: Silva, Alvelos, and Valério).
The SAV is very similar to the FMT, which is examined in~\cref{sec:02methods}.
In fact, the SAV may be seen as an FMT variant that uses POCs instead of BGCs.
The limitation to two- and three-stage problems comes from the cut-and-plate enumeration.
If the enumeration is not stopped at a specific stage, the SAV immediately supports unlimited stages.
Essentially, the proposed change is to: hybridise the FMT with the SAV, replacing BGCs with POCs \emph{only} when doing so cannot lead to loss of optimality for the unrestricted problem.
%Next section presents the implementation details for this change.

\subsection{Implementation details}

As seen in the last section, a POC (\emph{piece-oulining cut}) is prefered over a BGC (\emph{basic guillotine cut}) if it is guaranteed that replacing the latter by the former will not cause loss of optimality.
For the restricted problem, the typical set of horizontal (vertical) cutting positions is just the set of unique values in~\(l_i\) (\(w_i\)) for every piece type~\(i\) that fits into the plate.
Besides one corner case, each single guillotine cut at such positions may be replaced by the corresponding POC.
The corner case arises in cutting positions that come from a length (or width) value shared by two or more pieces.
In this case, a single guillotine cut needs to be replaced by two or more POCs, depending on how many pieces share the corresponding cutting position; otherwise, the model would lose the capability to produce that piece type.

For the unrestricted problem, the exact set of cutting positions often varies between different solving methods.
There are many discretisation procedures and reductions to be applied either after or during such discretisations.
The author will focus on the discretisations and reductions procedures employed by the formulations of~\citet{furini:2016} and~\citet{becker:2021} (this is, the FMT and BBA formulations).
The base discretisation employed by both FMT and BBA is straightforward: \(q\) is an horizontal (vertical) cutting position if, and only if, there is a demand-abiding linear combination of lengths (widths) from pieces that fit into the respective plate.
This cutting position set is a superset of the restricted set (from the last paragraph) and will be referred to as the \emph{base unrestricted set}.
Suppose a cutting position allows for the associated BGC to be replaced by (one or more) POCs without loss of optimality for the \emph{unrestricted problem}.
In that case, the cutting position (and, by extension, the BGC) is said to be \emph{replaceable}.
%The set of cutting positions for which the associated single guillotine cuts may be replaced by POCs, without loss of optimality in the context of the \emph{unrestricted problem}, will be referred to as the \emph{replaceable set} henceforth.

A cutting position must meet two conditions to be deemed replaceable.
The first condition is that a cutting position of the same orientation for the same plate exists in the restricted set.
This first condition is necessary because, otherwise, the cut is not outlining a piece, i.e., there is no corresponding piece type to be extracted from the first child plate.
The second condition is that such horizontal (vertical) cutting positions cannot be obtainable by a demand-abiding linear combination of two or more piece lengths (widths), considering only the pieces that fit into the respective plate.
This second condition is necessary because, otherwise, the replaced cut could be necessary for the only optimal cutting pattern of an instance of the unrestricted problem.
An example of this situation can be seen in~\cref{fig:distinctions_restricted_unrestricted} (the middle pattern, i.e., Position-only Restricted).
The middle vertical cut matches a piece width (i.e., it satisfied the first condition); however, if it were replaced by a POC associated with the square piece, it would be impossible to obtain the unrestricted optimal solution (that needs a BGC at the same position).
%the sum of lengths (widths) for a piece multiset in which (a) each piece type multiplity respects the demand, (b) each piece type fits into the first child plate, and (c) the multiset has a cardinality of two or more.
% NOTE: probably we need an image with three diagrams, two of them equal to the introduction ones and a third image showing the possibility of the cut if there is an out-of-the-pattern piece type that has the same size as two other summed.

The two reductions proposed in~\citet{furini:2016}, \emph{Cut-Position} and \emph{Redundant-Cut}, cause little change to the replaceable cutting positions. % Both reductions are briefly described at the start of~\cref{sec:var_enum}.
The only cutting positions removed by \emph{Cut-Position} are the ones not in the restricted set and, therefore, not replaceable. % Moreover, if the cutting position set of a plate is reduced by \emph{Cut-Position} and the kept positions are all replaced with POCs, then that plate and any plate strictly smaller than it will, in fact, be solved by the SAV formulation instead of the FMT formulation.
\emph{Redundant-Cut} may remove a replaceable cutting position.
However, as the name says, it only removes cuts that have an alternative cutting position that can be employed for the same effect (which is kept).
The predicted alternative cutting position will always be replaceable too, and replacing it with one or more POCs never requires adding back the cuts removed by Redundant-Cut.
Also, the BBA formulation never has trim cuts like those removed by Redundant-Cut, so this enhancement is superseded by it.

BBA adds extraction variables and reduces the base unrestricted set to only the cutting positions up to the midplate.
The extraction variables can be seen as POCs in which both top and right residual plates are guaranteed to be waste; therefore, extractions are not subject to be replaced by POCs.
BBA requires us to differentiate between \emph{binding} and \emph{non-binding} POCs.
A POC is \emph{non-binding} if the piece-sized plate it obtains may be regarded as waste; conversely, if the piece-sized plate must be sold as a piece, then the POC is \emph{binding}.
A \emph{binding} POC cannot be employed if an extra copy of the associated piece type would lead to disrespecting the demand constraint.
If replaceable cuts in the BBA formulation are replaced by \emph{binding} POCs, then there are cases in which loss of optimality occurs.
The cause of this loss of optimality is that, in BBA, a replaceable cut may be required by an optimal solution even if there is no demand for the associated piece.
These seemingly unnecessary cuts aim to reduce the plate size until a large piece can be obtained from the plate through an extraction variable.
A complete example follows.

\begin{example}{Hybridised BBA with binding cuts loses optimality.}
Consider the following G2KP instance: \(L = 100\), \(W = 100\), \(l = [100, 100]\), \(w = [1, 51]\), \(u = [1, 1]\), and \(p = [1, 1]\).
The optimal solution clearly must contain the only available copy of each of the two piece types.
In BBA, there is no cut after the midplate; consequently, a vertical cut at position~\(51\) is ruled out.
The only possibility is a vertical cut at position~\(1\) for which the first child plate could be immediately sold as the single copy of the first piece type.
The second child plate (\(100\)x\(99\)) also does not have an extraction variable for the immediate extraction of the second piece type (\(100\)x\(51\)).
The BBA determines that for an extraction variable to exist ``[...] the plate cannot fit an extra piece (of any type).'' and the first piece type fits together with the second in the \(100\)x\(99\) plate.
Again, a vertical cut at position~\(51\) is unavailable because it happens after midplate.
Consequently, BBA forces the optimal solution to create~\(50\) plates of size~\(100\)x\(1\), one of which will be sold as a piece, and the rest considered waste.
The second child of the~\(50\)th (and last) cut has size~\(100\)x\(50\), and it can be sold as the second piece type because an extraction variable is now available (i.e., the previously quoted condition does not apply anymore).
The adoption of \emph{binding} POCs makes it impossible for BBA to obtain an optimal solution for this example.
The reason is that there are not~\(50\) copies of the first piece type, but these would be needed by the~\(50\) binding piece-outlining cuts necessary to obtain an optimal solution.
The same problem does not arise if the POCs are not binding.
\end{example}

The corner case of two or more pieces sharing the same length/width needs to be considered in the unrestricted problem too, but with a subtle distinction.
In the restricted problem, replacing every single guillotine cut by POCs also brings the advantage of not needing an additional mechanism to enforce the problem definition (i.e., to guarantee piece extractions from the first child plates).
However, in the unrestricted problem, the choice between replacing a single guillotine cut by multiple POCs, or keeping it as a guillotine cut, is just a trade-off between model size and model symmetry.
Therefore, this work further distinguishes between two implementations of hybridisation.
The \emph{conservative} hybridisation substitutes each replaceable horizontal (vertical) BGC with one horizontal-first (vertical-first) POC that is associated with the single piece type that matches the length (width) of the cutting position (and that fits into the respective plate).
If two or more fitting piece types match the cutting position, the conservative hybridisation leaves the BGC unchanged.
The \emph{aggressive} hybridisation substitutes each replaceable horizontal (vertical) BGC with one horizontal-first (vertical-first) POC \emph{for each piece type} that matches its length (width) (and that fits into the respective plate).
%enhancement based on how they deal with the corner case of multiple pieces sharing length or width.
%The distinction between aggressive and conservative is made mostly during cut and plate enumeration, and depending on implementation details, the model formulation may be kept blind to it, i.e., only taking into account which cuts are BGCs and which are POCs.

We believe it is excessive to present the full formulation and implementation details for every combination of the FMT/BBA formulation with conservative/aggressive hybridisation and binding/non-binding POCs.
The experiments in the next section only consider the BBA with conservative/aggressive hybridisation and non-binding POCs.
The distinction between conservative and aggressive hybridisation is mostly made at the cut and plate enumeration; however, because of an unfortunate notation detail explained further, it is less troublesome to present an accurate formulation of the conservative hybridisation than the aggressive hybridisation.
In light of this, we chose to fully present the conservative hybridised BBA formulation with non-binding POCs.
The main differences in implementing other combinations are briefly discussed shortly after.

We start by explaining in detail the {\modelBecker}.
The {\modelBecker} employs the same two reductions of {\modelFMT} (Cut-Position and Redundant-Cut) described in~\cite{furini:2016}, and the cut and plate enumeration is the same except by the two enhancements described below.
The first enhancement is that every intermediary plate for which the length (width) is not a linear combination of the pieces length (width) have its dimensions \emph{reduced} to the closest linear combination; this combines multiple plates that could only pack the same set of pieces into a single plate type.
Such enhancement was already proposed in other contexts~\cite{alvarez:2009,dolatabadi:2012}.
The second enhancement is that {\modelFMT} considered cuts after the middle of a plate, while {\modelBecker} avoids enumerating any cuts after the middle of a plate.
This is done by replacing a variable set from~{\modelFMT}, by a set \(e_{ij}~\forall (i, j) \in E \subset \bar{J} \times J\), representing plate~\(j\) had piece~\(i\) extracted from it and sold.
An extraction variable exists (i.e., \(e_{ij} \in E\)) iff the plate~\(j\) dimensions do not allow to extract \emph{both} piece~\(i\) and another piece.

For convenience, we also define \(E_{i*} = \{ j : \exists~(i, j) \in E \}\) and \(E_{*j} = \{i : \exists~(i, j) \in E \}\).
The set \(O = \{h, v\}\) denotes the horizontal and vertical cut orientations.
The set \(Q_{jo}\) (\(\forall j \in J, o \in O\)) denotes the set of possible cuts (or cut positions) of orientation~\(o\) over plate~\(j\).
The plate~\(0 \in J\) is the original plate, and it may also be in~\(\bar{J}\), as there may exist a piece of the same size as the original plate.
The parameter~\(a\) is a byproduct of the plate enumeration process.
The value of \(a^o_{qkj}\) is the number of plates \(j \in J\) added to the stock if a cut of orientation~\(o \in O\) is carried out at position~\(q \in Q_{jo}\) of a plate~\(k \in J\).

In a valid solution, the value of \(x^o_{qj}\) is the number of times a plate~\(j \in J\) is cut with orientation~\(o \in O\) at position~\(q \in Q_{jo}\); i.e., how much flow is being transported by each edge coming from a plate \oldtext{vertice}\newtext{vertex}.
The plate~\(0 \in J\) is the original plate, and it may also be in~\(\bar{J}\), as there may exist a piece of the same size as the original plate.

\begin{align}
\bm{max.} &\sum_{(i, j) \in E} p_i e_{ij} \label{eq:objfun}\\
\bm{s.t.} &\specialcell{\sum_{o \in O}\sum_{q \in Q_{jo}} x^o_{qj} + \sum_{i \in E_{*j}} e_{ij} \leq \sum_{k \in J}\sum_{o \in O}\sum_{q \in Q_{ko}} a^o_{qkj} x^o_{qk} \hspace*{0.05\textwidth} \forall j \in J, j \neq 0,}\label{eq:plates_conservation}\\
%            & \specialcell{\sum_{o \in O}\sum_{q \in Q_{jo}} x^o_{qj} \leq \sum_{k \in J}\sum_{o \in O}\sum_{q \in Q_{ko}} a^o_{qkj} x^o_{qk} \hspace*{\fill} \forall j \in J\setminus\bar{J},}\label{eq:generic_plates_conservation}\\
	    & \specialcell{\sum_{o \in O}\sum_{q \in Q_{0o}} x^o_{q0} + \sum_{i \in E_{*0}} e_{i0} \leq 1 \hspace*{\fill},}\label{eq:just_one_original_plate}\\
            & \specialcell{\sum_{j \in E_{i*}} e_{ij} \leq u_i \hspace*{\fill} \forall i \in \bar{J},}\label{eq:demand_limit}\\
	    % TODO: fix equation below, the forall part is too long and clashes with the long equation in the first line
	    & \specialcell{x^o_{qj} \in \mathbb{N}^0 \hspace*{\fill} \forall j \in J, o \in O, q \in Q_{jo},}\label{eq:trivial_x}\\
            & \specialcell{e_{ij} \in \mathbb{N}^0 \hspace*{\fill} \forall (i, j) \in E.}\label{eq:trivial_e}
\end{align}

The objective function maximizes the profit of the extracted pieces~\eqref{eq:objfun}.
Constraint~\eqref{eq:plates_conservation} guarantees that for every plate~\(j\) that was further cut or had a piece extracted from it (left-hand side), there must be a cut making available a copy of such plate (right-hand side).
One copy of the original plate is available from the start~\eqref{eq:just_one_original_plate}, and it can be either have a piece directly extracted or be further cut.
The amount of extracted copies of some piece type must respect the demand for that piece type (a piece extracted is a piece sold)~\eqref{eq:demand_limit}.
Finally, the domain of all variables is the non-negative integers~\eqref{eq:trivial_x}-\eqref{eq:trivial_e}.

The \emph{conservative hybridised BBA formulation with non-binding} POCs requires a new set of variables, a new set of constraints, a new parameter, and some minor changes to the objective function and some of the existing constraints.
Both the new set of variables and the new set of constraints are bounded by \(|\bar{J}|\) and, therefore, cause only a small relative increase to the model size of a non-trivial instance.
The notation for the new variable and parameter set follows:

\begin{description}
\item [\(s_i\)] \(\forall i \in \bar{J}\) -- Integer variable. Indicates how many piece-sized plates obtained by POCs associated with piece type~\(i\) were sold as pieces of type~\(i\). By \emph{sold} the author means they contributed to the objective function and were accounted for by the demand constraint.
\item [\(h^o_{qji}\)] \(\forall o \in O, j \in J, q \in Q_{jo}, i \in \bar{J}\) -- Binary parameter. Byproduct of the cut and plate enumeration. It has value one if cut~\(x^o_{qj}\) is a POC that produces a piece-sized plate corresponding to piece~\(i\); zero otherwise.%For a cut~\(x^o_{qj}\), \(\sum_{i\in\bar{J}} h^o_{qji}\) is zero if the cut is a BGC, and one if the cut is a POC.
\end{description}

The parameter \(a^o_{qkj}\) from \{\modelBecker\} is exactly the same for BGCs and has a slightly different meaning for POCs.
The difference is that the \(j\) (obtained child plate) is always either the top or right residual (i.e., the POC version of the first and second child) and that both \(o\) (orientation) and \(q\) (cutting position) refer only to the first constituting cut of a POC; the meaning of \(k\) (parent plate) is left unchanged.
The set of variables representing cuts (\(x^o_{qj}\)) also does not need change, as \(h^o_{qji}\) fills the need to identify POCs and their associated piece types. Consequently, the constraints~\eqref{eq:plates_conservation} and~\eqref{eq:just_one_original_plate} presented below are the same as the non-hybridised formulation.

The constraint~\eqref{eq:piece_sized_plates} guarantees each piece-sized plate available~(\(s_i\)) comes from an actual POC.
The remaining changes consist into adding \(s_i\) to the demand constraint~\eqref{eq:hyb_demand} (which avoids overproduction without prohibiting the POCs themselves) and to the objective function~\eqref{eq:hyb_obj} (which allows piece-sized plates to be sold).

% TODO: check if labels are correct

\begin{align}
\bm{max.} &\sum_{(i, j) \in E} p_i e_{ij} + \sum_{i \in \bar{J}} p_i s_i \label{eq:hyb_obj}\\
\bm{s.t.} &\specialcell{\sum_{o \in O}\sum_{q \in Q_{jo}} x^o_{qj} + \sum_{i \in E_{*j}} e_{ij} \leq \sum_{k \in J}\sum_{o \in O}\sum_{q \in Q_{ko}} a^o_{qkj} x^o_{qk} \hspace*{0.05\textwidth} \forall j \in J, j \neq 0,}\tag{\ref{eq:plates_conservation}}\\
	    & \specialcell{\sum_{o \in O}\sum_{q \in Q_{0o}} x^o_{q0} + \sum_{i \in E_{*0}} e_{i0} \leq 1 \hspace*{\fill},}\tag{\ref{eq:just_one_original_plate}}\\
            & \specialcell{s_i \leq \sum_{j \in J}\sum_{o \in O}\sum_{q \in Q_{jo}} h^o_{qji} x^o_{qj} \hspace*{\fill} \forall i \in \bar{J},}\label{eq:piece_sized_plates}\\%\tag{\ref{eq:piece_sized_plates}}\\
            & \specialcell{s_i + \sum_{j \in E_{i*}} e_{ij} \leq u_i \hspace*{\fill} \forall i \in \bar{J},}\label{eq:hyb_demand}\\
	    & \specialcell{x^o_{qj} \in \mathbb{N}^0 \hspace*{\fill} \forall j \in J, o \in O, q \in Q_{jo},}\tag{\ref{eq:trivial_x}}\\
            & \specialcell{e_{ij} \in \mathbb{N}^0 \hspace*{\fill} \forall (i, j) \in E}\tag{\ref{eq:trivial_e}}\\
            & \specialcell{s_{i} \in \mathbb{N}^0 \hspace*{\fill} \forall i \in \bar{J}.}\label{eq:trivial_s}
\end{align}

The aforementioned unfortunate notation detail is the incapability of denoting two or more different cuts~\(x^o_{qj}\) with the same orientation~\(o\) and the same cutting position~\(q\) over the same plate~\(j\).
Therefore, if the aggressive hybridisation replaces a BGC with two or more POCs, then the notation does not allow us to differentiate between them. The~\(a^o_{qkj}\) parameter also needs to change, as it suffers from the same problem.
The aggressive hybridisation code deals with this problem by having unique single indexes for each cut and reverse indexes from each cut property (like orientation or cutting position) to the cuts themselves; this way, the cuts are not limited to the uniqueness of some property combination.

A trivial way to change the presented formulation to use \emph{binding cuts} is to change the constraint set~\eqref{eq:piece_sized_plates} to require equality.
However, the binding cuts can also be implemented without the new variable and constraint sets.
The term \(s_i\) could just be replaced by \(\sum_{j \in J}\sum_{o \in O}\sum_{q \in Q_{jo}} h^o_{qji} x^o_{qj}\) in both the objective function and the demand constraint.
Both mentioned ways to implement binding cuts work on the FMT formulation, which does not have the same loss of optimality problem as the BBA.

\subsection{Experimental results}

In these experiments, for reasons explained further ahead, each instance was solved ten times with ten distinct solver seeds.
The BBA configuration includes the applicable reductions previously discussed (i.e., Cut-Position and Plate-Size Normalisation).
The barrier algorithm was used to solve the root node as usual.
Only the Gurobi solver is used in these experiments.
No runs ended in timeout.
The computer setup, as well as the Julia and Gurobi versions/parameters, are the same as described in~\cref{sec:02methods}, but the model was built and solved in the same process (i.e., there was no writing and reading from MPS file), and no time limit was enforced.
Three variants are scrutinised: no hybridisation (N. H.), conservative hybridisation (C. H., avoids increasing model size), and aggressive hybridisation (A. H., always hybridise, even if it leads to an increase of the model size).
The first dataset considered is FMT59 (solved as G2KP), and the second is CJCM (solved as G2OPP\footnote{The Guillotine 2D Orthogonal Packing Problem is the decision problem related to the G2KP. It consists in proving that is possible (or impossible) to pack \emph{every} piece into the original plate. The adaptation con}).%; more details on these datasets can be found in~\cref{sec:datasets}.

\Cref{tab:g2kp_hyb_summary} shows that both C. H. and A. H. had a similar impact on the total solving time (i.e., a reduction of \(\approx\)20\%).
A. H. had slightly better timings despite the considerable increase in the number of cuts.
C. H. slightly reduces the number of cuts.
Both C. H. and A. H. have almost no effect on the number of plates (or extractions variables).
The percentage of hybridised cuts (\emph{h \%}) and hybridised cuts with just one residual (\emph{k \%}) show that the new reductions changed a very significant part of the models.
The number of instances with the lowest averages shows that N. H. is the best option for most instances.

\begin{table}
\caption{
Summary of hybridisation impact over BBA formulation and FMT59 dataset.
\emph{T. T.} (Total Time) -- sum of the mean time of all instances, in seconds; \emph{\(\Delta\) B. T.} (Distance from the Best Time) -- sum of the difference in mean time between the respective variant and the variant with the lowest mean time for the same instance, in seconds, i.e., if all variants ran in parallel and had average time, how much time the runs of the respective variant would spend after another thread already finished; \emph{\#b} (best) -- number of instances in which the respective variant had the lowest (best) average time among the variants; \emph{\#extr.} -- total number of extraction variables (considering one model per instance); \emph{\#cuts.} -- total number of cut variables (considering one model per instance); \emph{h \%} -- percentage of \#cuts that were hybridised; \emph{k \%} -- percentage of \#cuts that were not only hybridised but also discarded the second child of the second constituting cut as waste, i.e., the POC resulted in the piece-sized plate and \emph{one} other plate; \emph{\#plates} -- total number of plates (considering one model per instance).
}
\label{tab:g2kp_hyb_summary}
\begin{center}
\begin{tabular}{lrrrrrrrr}
\hline\hline
\textbf{Variant} & \textbf{T. T.} & \textbf{\(\Delta\) B. T.} & \textbf{\#b} & \textbf{\#extr.} & \textbf{\#cuts} & \textbf{h \%} & \textbf{k \%} & \textbf{\#plates} \\\hline
N. H. & 6,681 & 1,703 & 27 & 186,536 & 2,498,801 & -- & -- & 113,822 \\
C. H. & 5,468 & 489 & 22 & 184,067 & 2,496,421 & 41 & 18 & 113,373 \\
A. H. & 5,447 & 469 & 10 & 184,050 & 3,021,911 & 67 & 28 & 113,366 \\\hline\hline
\end{tabular}
\end{center}
\end{table}

A closer look into the data, see \Cref{tab:g2kp_hyb_selected_instances}, reveals that most time difference comes from a few hard instances.
In fact, the instance Hchl4s alone is responsible for most of the difference, with okp2 having about half its relevance and the rest of the instances considerably less impact.
The number of variables hybridised (\emph{H} columns) does not seem a good indicator of how impacted the solving times will be.
However, if N. H. spends most of the time solving the root node (low Non-Root \%), C. H. and A. H. generally do not bring great time improvements.
As the most significant reductions often occur in instances that spend less than 1\% of the time in the root node, the time distribution does not change significantly.
An exception is CHL1s which shows that C. H. seems to impact not the time at the root node but the time at the B\&B, as expected from a symmetry breaking-enhancement.

\begin{table}[!ht]
\caption{
Impact of BBA hybridisation in FMT59 instances taking more than 10s.
\emph{H (\%)} -- the percentage of all variables (i.e., cut and extraction) that were hybridised for C. H. and A. H.; \emph{Mean Time (s/\%)} -- mean time spent to solve the instance, in seconds for N. H., and in a percentage relative to N. H. for both C. H. and A. H.; \emph{CV} -- coefficient of variation (also known as relative standard deviation) is the standard deviation for N. H., C. H., and A. H., divided by their respective means (CV is always a percentage); \emph{Non-Root (\%)} -- the percentage of the total time which was \emph{not} spent solving the root node.
}
\label{tab:g2kp_hyb_selected_instances}
\begin{center}
%\resizebox{!}{.77\height}{%
\begin{tabular}{lrrrrrrrrrrr}
\hline\hline
& \multicolumn{2}{c}{H (\%)} & \multicolumn{3}{c}{Mean Time (s/\%)} & \multicolumn{3}{c}{CV (\%)} & \multicolumn{3}{c}{Non-Root (\%)} \\\cmidrule(lr){2-3}\cmidrule(lr){4-6}\cmidrule(lr){7-9}\cmidrule(lr){10-12}
Inst. & C & A & N (s) & C (\%) & A (\%) & N & C & A & N & C & A \\\hline\hline
Hchl4s & 46 & 57 & 3,657 & 80 & \bestcolumnemph{69} & 76 & 35 & 25 & >99 & >99 & >99 \\
okp2 & 22 & 22 & 1,844 & \bestcolumnemph{77} & 88 & 21 & 19 & 44 & >99 & >99 & >99 \\
Hchl7s & 50 & 77 & 428 & \bestcolumnemph{100} & 134 & 18 & 26 & 19 & 25 & 36 & 44 \\
okp3 & 33 & 49 & \bestcolumnemph{209} & 113 & 122 & 29 & 38 & 35 & >99 & >99 & >99 \\
Hchl8s & 17 & 35 & 253 & 68 & \bestcolumnemph{48} & 73 & 45 & 44 & >99 & >99 & >99 \\
Hchl3s & 46 & 57 & 39 & \bestcolumnemph{93} & 130 & 11 & 21 & 85 & 80 & 82 & 87 \\
Hchl2 & 25 & 77 & 45 & \bestcolumnemph{89} & 144 & 5 & 8 & 18 & 49 & 50 & 65 \\
CHL6 & 45 & 68 & 39 & \bestcolumnemph{91} & 98 & 13 & 15 & 14 & 48 & 46 & 44 \\
CHL7 & 23 & 78 & \bestcolumnemph{30} & 105 & 116 & 8 & 8 & 5 & 26 & 27 & 44 \\
Hchl6s & 51 & 80 & \bestcolumnemph{36} & 103 & 103 & 3 & 5 & 3 & 14 & 18 & 27 \\
CHL1 & 32 & 64 & \bestcolumnemph{26} & 115 & 120 & 14 & 16 & 14 & 68 & 71 & 72 \\
CHL1s & 32 & 64 & 21 & \bestcolumnemph{75} & 121 & 14 & 13 & 6 & 62 & 39 & 61 \\
okp5 & 11 & 12 & 12 & \bestcolumnemph{97} & 99 & 1 & 2 & 2 & 19 & 21 & 21 \\\hline\hline
\end{tabular}
%} % resizebox
\end{center}
\end{table}

The coefficient of variation of the analysed instances reveals the reason for multiple runs with distinct seeds: the difference between two runs of the same variant but distinct seeds is often larger than the difference between the means of two distinct variants.
Intuitively, breaking symmetries should reduce the variance of the timings.
By cutting symmetric branches, there is less opportunity for a solver seed to traverse multiple equivalent branches with a good relaxation (but bad primal) before finding a primal solution that cuts all such branches.
In fact, when C. H. and A. H. achieve a considerable (20\% or more) reduction of the mean time, the coefficient of variation (which is relative to the mean time) generally shows a reduction.
However, a more general effect, i.e., a higher percentage of model hybridisation (H\%) leading to lower CV (or mean time), is not observed.
Exactly \emph{which} variables were hybridised probably have more impact than \emph{how many} variables.
Finally, the reduction of variance, while positive if the objective is to compare solution methods, may be unwanted when solving the same problem in parallel.
For example, if two methods have similar mean times, the method with the most variance will probably have a thread find the optimal solution first.

In general, C. H. either had a negligible difference from N. H. or provided some considerable benefit (especially for instances with longer running times).
For solving mostly small instances, the extra complexity brought to the formulation may not be worthwhile, but the change does not bring much risk of worsening the results.
The A. H. has the best reduction of mean time and CV for both Hchl4s and Hchl8s, but it has a consistently bad performance for small instances with many pieces sharing the same length or width.
There is no clear class of instances for which it can consistently outperform C. H. (or N. H.).

%% file: piece_outlining_cut.tex
\begin{tikzpicture}[scale=0.145]
\def\piececolor{gray!20}
\def\labelxshift{12.5}
\def\labelyshift{0}
\def\labelfontsize{\normalsize}
\begin{scope}[shift={(0, 0)}] % FIRST ROW
\begin{scope}[shift={(0, 0)}] % FIRST IMAGE

\fill[\piececolor] (0, 0) rectangle +(15, 10);
\draw[thick, black] (0, 10) -- (25, 10);
\draw[dashed, thick, black] (15, 0) -- (15, 10);
\draw (0,0) rectangle +(25, 25);

\node [align=center, font=\labelfontsize\selectfont] at (7.5, 5) {\labelfontsize piece}; %{\labelfontsize piece-sized \\ \labelfontsize plate};
\node [align=center, font=\labelfontsize\selectfont] at (12.5, 17.5) {\labelfontsize top residual};
\node [align=center, font=\labelfontsize\selectfont] at (20, 5) {right\\residual};

\node [below] at (\labelxshift, \labelyshift) {\labelfontsize Horizontal-first POC};
\end{scope}

\begin{scope}[shift={(30, 0)}] % SECOND IMAGE

\fill[\piececolor] (0, 0) rectangle +(15, 10);
\draw[dashed, thick, black] (0, 10) -- (15, 10);
\draw[thick, black] (15, 0) -- (15, 25);
\draw (0,0) rectangle +(25, 25);

\node [align=center, font=\labelfontsize\selectfont] at (7.5, 5) {\labelfontsize piece}; %{\labelfontsize piece-sized \\ \labelfontsize plate};
\node [align=center, font=\labelfontsize\selectfont] at (7.5, 17.5) {\labelfontsize top residual};
\node [align=center, font=\labelfontsize\selectfont] at (20, 12.5) {right\\residual};

\node [below] at (\labelxshift, \labelyshift) {\labelfontsize Vertical-first POC};
\end{scope}

\begin{scope}[shift={(60, 0)}] % THIRD IMAGE

\fill[\piececolor] (0, 0) rectangle +(15, 10);
\draw[thick, black] (0, 10) -- (15, 10);
%\draw[dashed, thick, black] (15, 0) -- (15, 10);
\draw (0,0) rectangle +(15, 25);

\node [align=center, font=\labelfontsize\selectfont] at (7.5, 5) {\labelfontsize piece}; %{\labelfontsize piece-sized \\ \labelfontsize plate};
\node [align=center, font=\labelfontsize\selectfont] at (7.5, 17.5) {\labelfontsize top residual};
\end{scope}

\begin{scope}[shift={(76, 0)}] % FOURTH IMAGE

\fill[\piececolor] (0, 0) rectangle +(15, 10);
%\draw[dashed, thick, black] (0, 10) -- (15, 10);
\draw[thick, black] (15, 0) -- (15, 10);
\draw (0,0) rectangle +(25, 10);

\node [align=center, font=\labelfontsize\selectfont] at (7.5, 5) {\labelfontsize piece}; %{\labelfontsize piece-sized \\ \labelfontsize plate};
\node [align=center, font=\labelfontsize\selectfont] at (20, 5) {right\\residual};
\end{scope}
\node [below, align=center, font=\labelfontsize\selectfont] at (71.5, 0) {\labelfontsize Single-cut POCs};

\end{scope}
\end{tikzpicture}

%% file: distinctions_restricted_unrestricted.tex
\begin{tikzpicture}[scale=0.175]
\def\piececolor{gray!20}
\def\labelxshift{12.5}
\def\labelyshift{0}
\def\labelfontsize{\normalsize}

\begin{scope}[shift={(0, 0)}] % FIRST ROW
\begin{scope}[shift={(0, 0)}] % FIRST IMAGE
\draw (0,0) rectangle +(25, 25);

%\draw[fill=\piececolor] (0,0) rectangle +(6, 19);
%\draw[fill=\piececolor] (6,0) rectangle +(5, 18);
\draw[fill=\piececolor] (14,0) rectangle +(6, 19);
\draw[fill=\piececolor] (20,0) rectangle +(5, 18);
\draw[fill=\piececolor] (0,0) rectangle +(9, 21);
\draw[fill=\piececolor] (9,0) rectangle +(5, 20);
%\draw[fill=\piececolor] (0,19) rectangle +(10, 6);
\draw[fill=\piececolor] (15,19) rectangle +(10, 6);
\draw[fill=\piececolor] (0,21) rectangle +(12, 3);

\draw[dashed, thick, black] (0, 21) -- (25, 21);
\draw[dashed, thick, black] (9, 0) -- (9, 25);

\node [below] at (\labelxshift, \labelyshift) {\labelfontsize Restricted};

\draw (0,0) rectangle +(25, 25);
\end{scope}

\begin{scope}[shift={(30, 0)}] % SECOND IMAGE
\draw (0,0) rectangle +(25, 25);

%\draw[fill=\piececolor] (0,0) rectangle +(6, 19);
%\draw[fill=\piececolor] (6,0) rectangle +(5, 18);
\draw[fill=\piececolor] (14,0) rectangle +(6, 19);
\draw[fill=\piececolor] (20,0) rectangle +(5, 18);
\draw[fill=\piececolor] (0,0) rectangle +(9, 21);
\draw[fill=\piececolor] (9,0) rectangle +(5, 20);
%\draw[fill=\piececolor] (0,19) rectangle +(10, 6);
\draw[fill=\piececolor] (15,19) rectangle +(10, 6);
\draw[fill=\piececolor] (0,21) rectangle +(12, 3);

\draw[dashed, very thick, black] (0.1,0.1) rectangle +(13.9, 13.9);
\draw[dashed, thick, black] (14, 0) -- (14, 25);

\node [below] at (\labelxshift, \labelyshift) {\labelfontsize Position-only Restricted};
\end{scope}

\begin{scope}[shift={(60, 0)}] % THIRD IMAGE
\draw (0,0) rectangle +(25, 25);

%\draw[fill=\piececolor] (0,0) rectangle +(6, 19);
%\draw[fill=\piececolor] (6,0) rectangle +(5, 18);
\draw[fill=\piececolor] (14,0) rectangle +(6, 19);
\draw[fill=\piececolor] (20,0) rectangle +(5, 18);
\draw[fill=\piececolor] (0,0) rectangle +(9, 21);
\draw[fill=\piececolor] (9,0) rectangle +(5, 20);
%\draw[fill=\piececolor] (0,19) rectangle +(10, 6);
\draw[fill=\piececolor] (15,19) rectangle +(10, 6);
\draw[fill=\piececolor] (0,21) rectangle +(12, 3);

\draw[dashed, thick, black] (14, 0) -- (14, 25);

\node [below] at (\labelxshift, \labelyshift) {\labelfontsize Unrestricted};
\end{scope}

\end{scope}
\end{tikzpicture}

%% file: conclusions.tex
\section{Conclusions and future work}

We believe our work gives future practicioners a solid base to start from.
It reviews and compares empirically all known formulations for the {\myproblem} in the literature.
The {\modelBecker} solved most selected instances, but fails to provide good solutions for the hardest instances; while {\modelHierarchical}, {\modelImplicit}, and {\modelOrigami} had varying deegres of success and, in general, they do not prove optimality but provide good quality solutions even for the hardest instances. The {\modelBCE}, {\modelGrid}, and {\modelFMT} had trouble with small instances and were discarded from further comparison.
We found out that Gurobi has a small but consistent advantage over CPLEX, and the choice between CPLEX or Gurobi should not be able to bias against a specific formulation.
Our studies also revealed an inconsistency of the T instance dataset, for which guillotined and non-guillotined optima should match, but it does not.
We also propose an hybridisation of {\modelBecker} which consistently improves its run time in long-running instances.
The hybridisation comes from combining {\modelBecker} with the model from~\citet{silva:2010} (for a more restricted problem) showing that still there are ways to combine the ideas present in the literature to improve on existing models.
Considering that {\modelBecker} is state of the art, and that our hybridisation seem to either have no impact or positive impact, then our hybridisation supersedes it (except by the extra complexity of implementation).

The MILP formulations for the~{\myproblem} are recent and we expect this work to incentive further work in the subject.
Some directions for future research follows.
The APT dataset is yet a challenge for all known formulations.
The results of {\modelBecker} and our hybridisation shows the value of incremental improvements.
For pseudo-polynomial formulations, reducing the model size is key, as upper bounds are already great.
On the other hand, more compact formulations would benefit from aggregating better upper bounds, especially for the rotation variant.